\documentclass[12pt]{amsart}

\usepackage{amsmath,amsthm,amsfonts,amscd,flafter,epsf}

\usepackage{graphicx} 

\usepackage{amssymb} 

\usepackage{verbatim}  

\usepackage{colonequals}  

\usepackage[all]{xy}

\usepackage{pdfsync}

\usepackage{hyperref}
 
\usepackage{color}
 
\title[{Finitely $\mathcal{F}$-amenable actions and Decomposition Complexity of Groups}]
{Finitely $\mathcal{F}$-amenable actions and  Decomposition Complexity of Groups} 

 \author[Nicas]{Andrew Nicas}
\address{Department of Mathematics and Statistics, McMaster University, Hamilton, Ontario, Canada L8S 4K1}
\email{nicas@mcmaster.ca}
 
\author[Rosenthal]{David Rosenthal}
\address{Department of Mathematics and Computer Science, St.\ John's University, 8000 Utopia
Pkwy, Queens, NY 11439, USA}
\email{rosenthd@stjohns.edu}


\numberwithin{equation}{section}
    
\begin{document}

\begin{abstract}
In his work on the Farrell-Jones Conjecture,
Arthur Bartels introduced the concept of  a ``finitely $\mathcal{F}$-amenable'' group action, where $\mathcal{F}$ is a family of subgroups.
We show how a finitely $\mathcal{F}$-amenable action of a countable group $G$ on a compact metric space,
where the asymptotic dimensions of the elements of $\mathcal{F}$ are bounded from above,
gives an upper  bound for the asymptotic dimension of $G$
viewed as a metric space with a proper left invariant metric.
We  generalize this to families $\mathcal{F}$ whose elements are contained in a collection, $\mathfrak{C}$, of metric families  that satisfies some basic permanence properties:
If $G$ is a countable group and each element of $\mathcal{F}$ belongs to $\mathfrak{C}$ and there exists a finitely $\mathcal{F}$-amenable action of $G$ on a compact metrizable space,
then $G$ is in $\mathfrak{C}$.
Examples of such collections of metric families include:  metric families with weak finite decomposition complexity, exact metric families, and metric families that coarsely embed into Hilbert space.
\end{abstract}

\maketitle

\baselineskip 18pt


\newtheorem{theorem}{Theorem}[section]
\newtheorem*{thm}{Theorem}
\newtheorem{lemma}[theorem]{Lemma}
\newtheorem{proposition}[theorem]{Proposition}
\newtheorem{corollary}[theorem]{Corollary}
\newtheorem*{cor}{Corollary}
\newtheorem{claim}[theorem]{Claim}
\newtheorem{question}[theorem]{Question}

\theoremstyle{definition}
\newtheorem{definition}[theorem]{Definition}

\newtheorem*{defi}{Definition}

\theoremstyle{definition}
\newtheorem{remark}[theorem]{Remark}

\theoremstyle{remark}
\newtheorem{notation}[theorem]{Notation}

\theoremstyle{definition}
\newtheorem{example}[theorem]{Example}

\theoremstyle{definition}
\newtheorem{assumption}[theorem]{Assumption}

\newtheorem*{conditionA}{Condition (A)}
\newtheorem*{conditionB}{Condition (B)}
\newtheorem*{conditionC}{Condition (C)}

\newcommand{\ra}{{\rightarrow}}
\newcommand{\bs}{{\backslash}}
\newcommand{\lra}{{\longrightarrow}}


\newcommand{\asdim}{\operatorname{asdim}}
\newcommand{\hyperdim}{\operatorname{hyperdim}}
\newcommand{\whyperdim}{\operatorname{w-hyperdim}}
\newcommand{\id}{\operatorname{id}}
\newcommand{\supp}{\operatorname{supp}}
\newcommand{\st}{\operatorname{star}}
\newcommand{\ev}{\operatorname{ev}}
\newcommand{\mesh}{\operatorname{mesh}}
\newcommand{\diam}{\operatorname{diam}}
\newcommand{\catz}{\operatorname{CAT}(0)}

\newcommand{\tb}{{\large \textbullet}}

\newcommand{\Zhalf}{{\mathbb Z}[1/2]}


\newcommand{\cA}{\mathcal{A}}
\newcommand{\cB}{\mathcal{B}}
\newcommand{\cC}{\mathcal{C}}
\newcommand{\cD}{\mathcal{D}}
\newcommand{\cE}{{\mathcal E}}
\newcommand{\cF}{{\mathcal F}}
\newcommand{\cG}{{\mathcal G}}
\newcommand{\cH}{{\mathcal H}}
\newcommand{\cJ}{{\mathcal J}}
\newcommand{\calo}{\mathcal{O}}
\newcommand{\cP}{{\mathcal P}}
\newcommand{\cT}{{\mathcal T}}
\newcommand{\cU}{{\mathcal U}}
\newcommand{\cV}{{\mathcal V}}
\newcommand{\cW}{{\mathcal W}}
\newcommand{\cX}{{\mathcal X}}
\newcommand{\cY}{{\mathcal Y}}
\newcommand{\cZ}{{\mathcal Z}}

\newcommand{\fA}{\mathfrak{A}}
\newcommand{\fB}{\mathfrak{B}}
\newcommand{\fC}{\mathfrak{C}}
\newcommand{\fD}{\mathfrak{D}}
\newcommand{\wD}{w\mathfrak{D}}
\newcommand{\fE}{\mathfrak{E}}
\newcommand{\fH}{\mathfrak{H}}
\newcommand{\fL}{\mathfrak{L}}
\newcommand{\N}{\mathfrak{N}}
\newcommand{\fU}{\mathfrak{U}}


\newcommand{\tS}{{\widetilde S}}


\newcommand{\B}{{\rm B}}
\newcommand{\E}{{\rm E}}

\newcommand{\K}{{\mathbf K}}
\newcommand{\HH}{{\mathbf{HH}}}

\newcommand{\bN}{{\mathbb N}}
\newcommand{\R}{{\mathbb R}}
\newcommand{\Z}{{\mathbb Z}}




\section{Introduction}

The celebrated {\it Farrell-Jones Conjecture} asserts that certain ``assembly maps'' are isomorphisms. 
This conjecture is central to the modern study of high dimensional topology, see
\cite{Luck-ICM} and \cite{Bartels-ICM} for an overview.
Building on his approach in~\cite{Bartels1},  Bartels formulated the following concept of a ``finitely $\mathcal{F}$-amenable action''  in order to 
relate some of the key geometric conditions used to establish the Farrell-Jones Conjecture for many classes of groups to various notions of amenability arising in geometric group theory and analysis (\cite[Remark 0.4]{Bartels17}).

\begin{defi}[{\cite[Definition 0.1]{Bartels17}}]
Let $\cF$ be a family of subgroups of $G$ and let $N$ be a non-negative integer.
A $G$-action on a space $X$ is {\it $N$-$\cF$-amenable} if for any finite subset $S$ of $G$ there exists an open $\cF$-cover $\cU$ of $G\times X$ (equipped with the diagonal $G$-action) such that:
\begin{enumerate}
	\item the dimension of $\cU$ is at most $N$; and 
	\item for all $x\in X$ there is a $U \in \cU$ with $S \times \{x\} \subseteq U$.
\end{enumerate}
A $G$-action is called {\it finitely $\cF$-amenable} if it is $N$-$\cF$-amenable for some $N$.
\end{defi}

Bartels employed finitely $\cF$-amenable actions to elucidate the conditions used by Bartels, L\"{u}ck and Reich in~\cite{BLR08} to prove the Farrell-Jones Conjecture for word hyperbolic groups. 
In particular, he showed that if $G$ is a group that admits a finitely $\cF$-amenable action  
on a compact, finite-dimensional, contractible ANR,
then $G$ satisfies the $K$-theoretic Farrell-Jones Conjecture relative to $\cF$, \cite{Bartels17}.
A similar statement holds for the $L$-theoretic Farrell-Jones Conjecture.
Bartels and Bestvina established the Farrell-Jones Conjecture for mapping class groups by showing that the action of a mapping class group on the Thurston compactification of Teichm\"{u}ller space is finitely $\cF$-amenable,
where $\cF$ is the family of virtual point stabilizers, \cite{Bartels-Bestvina}.
The notion of finite $\cF$-amenability has also been studied by Sawicki~\cite{sawicki} in the guise of {\it equivariant asymptotic dimension}.

Our goal is to study the coarse geometric applications of finitely $\mathcal{F}$-amenable actions.
Willet and Yu observed (see \cite[Remark 1.3.5]{Bartels1})
that if there is a uniform bound on the asymptotic dimension of the groups in $\cF$, then a group $G$ that admits a finitely \linebreak 
$\cF$-amenable action on a compact metrizable space $X$ must have finite asymptotic dimension.
Motivated by this observation, we establish the following theorem using methods that allow us to extend it in a natural manner to a more general setting.

\begin{thm}[Theorem~\ref{cor:NF-amenable}]
Let $G$ be a countable group and $\cF$ be a family of subgroups of $G$. 
	If there exists an $N$-$\cF$-amenable action of $G$ on a compact metrizable space $X$ and $\asdim(F)\leq k$ for each $F\in \cF$, then $\asdim(G)\leq N+k$.	
\end{thm}

In this paper, 
we view a countable group $G$ as a metric space equipped with a proper left invariant metric,
where ``proper'' means that balls in $G$ have finitely many elements.
If $\rho$ and $\rho'$ are proper left invariant metrics on $G$ then the identity map $(G, \rho) \rightarrow (G, \rho')$
is a coarse equivalence by
\cite[Proposition 1.1]{DranishnikovSmith}.
Consequently, the asymptotic dimension of $G$, or any coarse property of $G$, 
is independent of the choice of  proper left invariant metric.

A {\it metric family} is a set whose elements are metric spaces.  
A {\it permanence property} of a collection $\fC$ of metric families is an operation that when applied to members of $\fC$ yields another member of $\fC$.
Our proof of the above theorem generalizes to families of subgroups whose elements are contained in a collection of metric families that satisfies some basic permanence properties.

\begin{thm}[Theorem~\ref{cor:NF-amenable-general}]
	Let $G$ be a countable group, $\cF$ be a family of subgroups of $G$, and $\fC$ be a collection of metric families satisfying Coarse Permanence~(\ref{def:closed-under-embeddings}), Finite Amalgamation Permanence~(\ref{def:amalgam}), and Finite Union Permanence~(\ref{def:finunion}). 
	
	 If there exists an $N$-$\cF$-amenable action of $G$ on a compact metrizable space and each $F\in \cF$ belongs to $\fC$, then $G$ is $N$-decomposable over $\fC$.	 If $\fC$ is also stable under weak decomposition, then $G$ is in $\fC$.
\end{thm}

In the case $\fC$ is the collection of metric families with asymptotic dimension at most $k$, Theorem~\ref{cor:NF-amenable-general} reduces to Theorem~\ref{cor:NF-amenable}. This depends on the fact that if a metric family $\cX$ $N$-decomposes over the collection of metric families with asymptotic dimension at most $k$, then $\asdim(\cX)$ is at most $N+k$, which we prove in Theorem~\ref{thm:asdim-pullback}.

Theorem~\ref{cor:NF-amenable-general} also applies to the collection $\fD$, of metric families with {\it finite decomposition complexity} (abbreviated to ``FDC"),
and to the collection $\wD$, of metric families with {\it weak finite decomposition complexity} (abbreviated to ``weak FDC"),
concepts introduced by Guentner, Tessera and Yu in \cite{Guentner_Tessera_Yu1, Guentner_Tessera_Yu2}.
A metric family $\cX$ is said to be {\it weakly decomposable} over a collection $\fC$ if,  for some non-negative integer $n$, $\cX$ is $n$-decomposable over $\fC$ (see Definition \ref{def:n-decompose})
and $\cX$ is {\it strongly decomposable} over $\fC$ if $\cX$ is $1$-decomposable over $\fC$.
The collection $\fD$ is  the smallest collection of metric families that contains all bounded metric families (that is, metric families whose elements have uniformly bounded diameters) and is stable under strong decomposition.
The collection $\wD$ is the smallest collection of metric families that contains all bounded metric families and is stable under weak decomposition.
FDC and weak FDC are interesting conditions because they have important topological consequences. 
A countable group with weak FDC satisfies the {\it Novikov Conjecture}, and a metric space with FDC and bounded geometry satisfies the {\it Bounded Borel Conjecture},~\cite{Guentner_Tessera_Yu1,Guentner_Tessera_Yu2}.
These results were obtained via an analysis of assembly maps in $L$-theory and topological $K$-theory.

Two other collections of metric families of importance for the Novikov Conjecture are:
\begin{itemize}
  \item $~\fE$, the collection of {\it exact metric families} (\cite[Definition 4.0.8]{Guentner_Tessera_Yu2}), and
  \item $~\fH$, the collection of metric families that are {\it coarsely embeddable into Hilbert space} (see Definition~\ref{def:coarse-embedding-Hilbert}).
\end{itemize}
It follows from \cite[Theorem 6.1]{STY} that a
countable group in $\fH$ satisfies the Novikov Conjecture.
Note that $\wD \subset \fE \subset \fH$.
Moreover, both $\fE$ and $\fH$ satisfy Coarse Permanence, Finite Amalgamation Permanence, and Finite Union Permanence. Since $\wD$, $\fE$ and $\fH$ are each stable under weak decomposition, we get the following corollary to Theorem~\ref{cor:NF-amenable-general}.

\begin{cor}[Corollary~\ref{cor:deh}]
Let $\fC$ be equal to $\wD$, $\fE$ or $\fH$. Let $G$ be a countable group and $\cF$ be a family of subgroups of $G$ such that each $F\in \cF$ belongs to $\fC$.
If there exists a finitely $\cF$-amenable action of $G$ on a compact metrizable space, then $G$ is in $\fC$.	
\end{cor}

In \cite{Bartels17}, Bartels proved that if a countable group $G$ is relatively hyperbolic with respect to peripheral subgroups $P_1,\dots, P_n$, then the action of $G$ on its boundary is finitely $\cP$-amenable, where $\cP$ is the family of subgroups of $G$ that are either virtually cyclic or subconjugated to one of the $P_i$'s. Thus, we obtain the following application of Theorem~\ref{cor:NF-amenable-general} to relatively hyperbolic groups.

\begin{thm}[Theorem~\ref{thm:rel_hyp}]
Let $G$ be a countable group that is relatively hyperbolic with respect to peripheral subgroups $P_1,\dots,P_n$, and let $\fC$ be a collection of metric families satisfying Coarse Permanence, Finite Amalgamation Permanence, and Finite Union Permanence. If $\fC$ contains $P_1,\dots,P_n$ and the infinite cyclic group $\Z$, then $G$ is $N$-decomposable over $\fC$ for some $N$. If $\fC$ is also stable under weak decomposition, then $G$ is in $\fC$.
\end{thm}

While Bartels and L\"{u}ck succeeded in verifying the Farrell-Jones Conjecture for $\catz$ groups by utilizing ``homotopy actions'' rather than finitely $\cF$-amenable actions~\cite{Bartels-Luck-a, Bartels-Luck-b},
it is still unknown if $\catz$ groups always have finite asymptotic dimension (Question \ref{catzero}).
Theorem~\ref{cor:NF-amenable} suggests a possible approach to this question.
Let $Y$ be a finite dimensional $\catz$ space on which the $\catz$ group $G$ acts geometrically.
Let $\cF$ be the family of virtual abelian subgroups of $G$.
If it is true that 
Caprace's {\it refined boundary}, $\partial_\infty^{\rm \, fine}\,Y$, as defined in \cite{Caprace-2009} has a compact metrizable topology for which the $G$-action on it  is finitely $\cF$-amenable (Question \ref{refined-boundary})
then $G$ has finite asymptotic dimension (Proposition~\ref{prop:catzero-is-true}).

{\bf Acknowledgements.} 
Andrew Nicas was partially supported by a grant from the Natural Sciences and Engineering Research Council of Canada.
David Rosenthal was partially supported by a grant from the Simons Foundation (\#524141). 
Rosenthal would also like to thank Wolfgang L\"{u}ck for his hospitality in Bonn, where portions of this project were completed,
and for support from L\"{u}ck's European Research Council Advanced Grant  ``KL2MG-interactions'' (\#662400). The authors would also like to thank the referee for many valuable comments.


\section{Decomposition Over a Collection of Metric Families}

In this section, we treat aspects of the coarse geometry of metric families needed for the proofs of our main technical results in Section~\ref{sec:main-thm}.
We recall the notion of {\it $n$-decomposition}, introduced by Guentner, Tessera and Yu as a generalization of finite asymptotic dimension,~\cite{Guentner_Tessera_Yu2}. 
Permanence properties of certain important collections of metric families are discussed.
We show that if a metric family $\cX$ is $n$-decomposable over the collection of metric families with asymptotic dimension at most $k$, then $\asdim(\cX)$ is at most $n+k$ (Theorem~\ref{thm:asdim-pullback}).
This result plays an important role in our proof of Theorem~\ref{thmA-asdim}.

\begin{definition}\label{def:r-decompose}
	Let $r>0$ and $n$ be a non-negative integer. The metric family $\cX$ is {\em $(r,n)$-decomposable} over the metric family $\cY$ if for every $X$ in $\cX$, $X=X_0 \cup X_1 \cup \cdots \cup X_n$ such that for each $i$ 
	$$X_i=\bigsqcup_{r\text{-disjoint}}X_{ij}$$ 
where each $X_{ij}$ is in $\cY$. 
\end{definition}

\begin{definition}\label{def:n-decompose}
	Let $n$ be a non-negative integer, and let $\fC$ be a collection of metric families. The metric family $\cX$ is {\em $n$-decomposable} over $\fC$ if for every $r>0$ $\cX$ is $(r,n)$-decomposable over some metric family $\cY$ in $\fC$. 
\end{definition}

Following~\cite{Guentner_Tessera_Yu2}, we say that $\cX$ is {\em weakly decomposable} over $\fC$ if $\cX$ is $n$-decomposable over $\fC$ for some non-negative integer $n$, and $\cX$ is {\em strongly decomposable} over $\fC$ if $\cX$ is $1$-decomposable over $\fC$.

\begin{definition}
	A metric family $\cZ$ is {\em bounded} if the diameters of the elements of $\cZ$ are uniformly bounded, that is, if $\diam(\cZ) = \sup\{\diam(Z)~|~Z \in \cZ\}<\infty$.
	The collection of all bounded metric families is denoted by $\fB$.
\end{definition}

\begin{example}\label{ex-asdim}
	Let $X$ be a metric space. The statement that the metric family $\{ X \}$ is $n$-decomposable over $\fB$ is equivalent to the statement that $\asdim(X) \leq n$.
\end{example}

The following definition is equivalent to Bell and Dranishnikov's definition of a collection of metric spaces having finite asymptotic dimension ``uniformly" (\cite[Section 1]{Bell_Dranish1}).

\begin{definition}\label{def:family-asdim}
	Let $n$ be a non-negative integer. The metric family $\cX$ has {\em asymptotic dimension at most $n$}, denoted $\asdim(\cX)\leq n$, if $\cX$ is $n$-decomposable over $\fB$.
\end{definition}

\begin{definition}\label{def:FDC}
	Let $\fD$ be the smallest collection of metric families containing $\fB$ that is closed under strong decomposition, and let $\wD$ be the smallest collection of metric families containing $\fB$ that is closed under weak decomposition. A metric family in $\fD$ is said to have {\em finite decomposition complexity} (abbreviated to ``FDC"), and a metric family in $\wD$ is said to have {\em weak finite decomposition complexity} (abbreviated to ``weak FDC").
\end{definition}

Clearly,  finite decomposition complexity implies weak finite decomposition complexity.
The converse is unknown.

\smallskip		
	Next, we recall some terminology introduced in~\cite{Guentner_Tessera_Yu2} that generalizes basic notions from the coarse geometry of metric spaces to metric families. 
	
	Let $\cX$ and $\cY$ be metric families.
	A {\it subfamily} $\cA$ of a metric family $\cY$ is a metric family such that every $A \in \cA$ is a metric subspace of some $Y \in \cY$.
	A {\it map of families}, $F:\cX \to \cY$, is a collection of functions $F=\{f:X \to Y\}$, where $X \in \cX$ and $Y \in \cY$, such that every $X \in \cX$ is the domain of at least one $f$ in $F$. 
The {\it inverse image} of the subfamily $\cA$ in $\cY$ under the map $F\colon \cX\rightarrow \cY$ is the subfamily of $\cX$ given by $F^{-1}(\cA)=\left\{ f^{-1}(A) \; | \; A\in \cA, f\in F \right\}$.

\begin{definition}\label{def:coarse-embedding}
\noindent
(1) A map of metric families, $F:\cX \to \cY$, is a {\em coarse embedding} if there exist non-decreasing functions $\delta, \varrho: [0,\infty) \to [0,\infty)$, with $\lim_{t\to \infty}\delta(t)=\infty=\lim_{t\to \infty}\varrho(t)$, such that for every $f:X \to Y$ in $F$ and every $x,y\in X$,
	\[ \delta\big(d_X(x,y)\big) \leq d_Y\big(f(x),f(y)\big) \leq \varrho\big(d_X(x,y)\big). \]

(2) A map of metric families, $F:\cX \to \cY$, is a {\em coarse equivalence} if for each $f:X \to Y$ in $F$ there is a map $g_f:Y\to X$ such that: 
	\begin{enumerate}
		\item[(i)] the collection $G=\{g_f\}$ is a coarse embedding from $\cY$ to $\cX$; and 
		\item[(ii)] the composites $f\circ g_f$ and $g_f \circ f$ are {\em uniformly close} to the identity maps $\id_Y$ and $\id_X$, respectively, in the sense that there is a constant $C>0$ with
	\[ d_Y\big(y,f\circ g_f(y)\big)\leq C \text{ and } d_X\big(x,g_f\circ f(x)\big)\leq C, \]
for every $f:X \to Y$ in $F$, $x\in X$, and $y\in Y$.
	\end{enumerate}
\end{definition}

\begin{definition}\label{def:closed-under-embeddings}
A collection of metric families, $\fC$, satisfies \emph{Coarse Permanence} if whenever $\cY\in\fC$ and $F\colon\cX\to\cY$ is a coarse embedding, then $\cX\in\fC$.
\end{definition}

Coarse Permanence is an important property for a collection of metric families to have when working with the notion of decomposition, as the following theorem illustrates. (A proof of it can be found in \cite[Theorem 9.13]{Nicas-Rosenthal}.)

\begin{theorem}\label{thm:cembed-decomp}
Let $\cX$ and $\cY$ be metric families, and let $\fC$ be a collection of metric families that satisfies Coarse Permanence. If $\cX$ coarsely embeds into $\cY$ and $\cY$ is $n$-decomposable over $\fC$, then $\cX$ is $n$-decomposable over $\fC$. In particular, if $\cX$ is coarsely equivalent to $\cY$, then $\cX$ is $n$-decomposable over $\fC$ if and only if $\cY$ is $n$-decomposable over~$\fC$. \qed
\end{theorem}

Since any two proper left invariant metrics on a countable group yield coarsely equivalent metric spaces (\cite[Proposition 1.1]{DranishnikovSmith}), we can make the following definition.

\begin{definition}\label{def:groups-in-C}
Let $\fC$ be a collection of metric families that satisfies Coarse Permanence. A countable group $G$ is said to be an element of $\fC$ if it is an element of $\fC$ with respect to any proper left invariant metric on $G$. A countable group $G$ is said to be $n$-decomposable over $\fC$ if it is $n$-decomposable over $\fC$ with respect to any proper left invariant metric on~$G$.
\end{definition}

Guentner, Tessera and Yu proved that both $\fD$ and $\wD$ satisfy Coarse Permanence \cite[Coarse Invariance 3.1.3]{Guentner_Tessera_Yu2}. It is straightforward to check that the following collections of metric families also satisfy Coarse Permanence.

\begin{example}\label{ex:ce} 
Collections of metric families that satisfy Coarse Permanence:
	\begin{enumerate}
		\item $\fB$, the collection of bounded metric families.
		\item $\fA$, the collection of metric families with finite asymptotic dimension.
		\item $\fA_n$, the collection of metric families with asymptotic dimension at most $n$.
		\item $\fE$, the collection of exact metric families (see Definition~\ref{def:exact}).
		\item $\fH$, the collection of metric families that are {\it coarsely embeddable into Hilbert space} (see Definition~\ref{def:coarse-embedding-Hilbert}).
		\end{enumerate}
\end{example}

\begin{definition}\label{def:exact}
	A metric family $\cX=\{X_\alpha \, | \, \alpha \in I\}$ is {\em exact} if for every $R>0$ and every $\varepsilon>0$, each $X_\alpha \in \cX$ admits a partition of unity $\{\phi_{U}^\alpha \}$ subordinate to a cover $\cU_\alpha$ such that:
	\begin{enumerate}
		\item[(i)] $\forall \alpha\in I$, $\forall x,x'\in X_\alpha$, $d_\alpha(x,x')\leq R \; \Rightarrow \; \displaystyle\sum_{U \in \cU_\alpha}\left| \phi_{U}^\alpha(x)-\phi_{U}^\alpha(x') \right|< \varepsilon$; and
		\item[(ii)] $\bigcup_{\alpha\in I}\cU_\alpha$ is a bounded metric family.
	\end{enumerate}
\end{definition}

\begin{definition}\label{def:coarse-embedding-Hilbert}
	A metric family $\cX=\{X_\alpha \, | \, \alpha \in I\}$ is {\em coarsely embeddable into Hilbert space}\footnote{The notion of a metric family that is coarsely embeddable into Hilbert space was introduced by Dadarlat and Guentner in~\cite{Dadarlat_Guentner1}, although they called it a ``family of metric spaces that is equi-uniformly embeddable."} if there is a family of Hilbert spaces $\cH=\{ H_\alpha  \, | \, \alpha \in I\}$ and a map of metric families $F=\{F_\alpha:X_\alpha \to H_\alpha \, | \, \alpha \in I\}$ such that $F:\cX \to \cH$ is a coarse embedding. The collection of all metric families that are coarsely embeddable into Hilbert space is denoted by $\fH$. 
\end{definition}

Exactness was introduced by Guentner and is closely related to Yu's {\it Property A} (see \cite[Section 5.2]{Guentner} for a discussion).
One of the goals of Yu's definition was to obtain a property that would imply coarse embeddability into Hilbert space and that is relatively easy to verify.
In particular, $\fE \subset \fH$. 
These concepts arose in conjunction with Yu's highly impactful work on the Novikov Conjecture, \cite{Yu, Yu2000, STY}.

\begin{definition}\label{def:amalgam}
	A collection of metric families, $\fC$, satisfies \emph{Finite Amalgamation Permanence} if the following holds. If $\cX=\bigcup_{i=1}^n\cX_i$ and each $\cX_i \in \fC$, then $\cX \in \fC$.
\end{definition}

It follows from~\cite[Fibering Theorem 3.1.4]{Guentner_Tessera_Yu2} that the collection of metric families with finite decomposition complexity, $\fD$, and the collection of metric families with weak finite decomposition complexity, $\wD$, satisfy Finite Amalgamation Permanence (also see \cite[Theorem 5.6]{KNR} for a proof). 

A basic property of asymptotic dimension is that if $A$ and $B$ are metric subspaces of some larger metric space, then $\asdim(A\cup B)={\rm max} \big\{\asdim(A),\asdim(B) \big\}$. This is also true for metric families, a property known as {\it Finite Union Permanence}.

\begin{definition}\label{def:finunion}
	A collection of metric families, $\fC$, satisfies \emph{Finite Union Permanence} if the following holds. Let $n \in \bN$ and let $\cX, \cX_1, \dots, \cX_n$ be metric families. If every $\cX_i \in \fC$ and for each $X\in \cX$ there exist $X_i\in \cX_i$, $1\leq i \leq n$, such that $X=\bigcup_{i=1}^n X_i$, then $\cX \in \fC$.
\end{definition}

We have the following elementary fact.

\begin{lemma}\label{lem:weak-decomp-finite-union}
Let $\fC$ be a collection of metric families that satisfies Finite Amalgamation Permanence and is stable under weak decomposition. Then $\fC$ satisfies Finite Union Permanence.
\end{lemma}

\begin{proof}
Let $\cX, \cX_1, \dots, \cX_n$ be metric families where every $\cX_i$ is in $\fC$ and for each $X\in \cX$ there exist $X_i\in \cX_i$, $1\leq i \leq n$, such that $X=\bigcup_{i=1}^n X_i$. Then, using Definition~\ref{def:r-decompose}, $\cX$ is $(r,n)$-decomposable over the metric family $\bigcup_{i=1}^n\cX_i$, which is in $\fC$ by Finite Amalgamation Permanence. Therefore, $\cX$ is $n$-decomposable over $\fC$ and thus, since $\fC$ is stable under weak decomposition, $\cX$ is in $\fC$. 
\end{proof}

The fact that the collection of metric families with asymptotic dimension at most $k$ satisfies Finite Union Permanence \cite[Theorem 6.3]{Guentner} gives the following straightforward application.

\begin{theorem}\label{thm:asdim-cosets-metric-spaces}
Let $G$ be a group, $X$ be a metric space with a $G$-invariant metric and $m$ be a natural number. Let $X_1, \dots, X_m$ be metric subspaces of $X$ and $\cX_m$ be the metric family $\big\{ \bigcup_{i=1}^m g_{i}X_{i} \;\big|\; g_{1}, \dots, g_m \in G \big\}$, where each $\bigcup_{i=1}^m g_{i}X_{i}$ is considered as a metric subspace of $X$. Then $\asdim(\cX_m) = {\rm max} \big\{ \asdim(X_i) \; \big| \; 1 \leq i \leq m \big\}.$
\qed	
\end{theorem}

Theorem~\ref{thm:asdim-cosets-metric-spaces} will be applied to groups in the proof of Theorem~\ref{thmA-asdim}. The precise statement we use is:

\begin{corollary}\label{thm:asdim-cosets-family}
	Let $G$ be a countable group with a proper left invariant metric, and let $m$ be a natural number. Let $H_1, \dots, H_m$ be subgroups of $G$ and $\cH_m$ be the metric family $\big\{ \bigcup_{i=1}^m g_{i}H_{i} \;\big|\; g_{1}, \dots, g_m \in G \big\}$, where each $\bigcup_{i=1}^m g_{i}H_{i}$ is considered as a metric subspace of $G$. Then $\asdim(\cH_m) = {\rm max} \big\{ \asdim(H_i) \; \big| \; 1 \leq i \leq m \big\}.$\qed	
\end{corollary}

The next result is needed to establish Theorem~\ref{thmA-asdim}. 

\begin{theorem}\label{thm:asdim-pullback}
If $\cX$ is a metric family that is $n$-decomposable over $\fA_m$ (the collection of metric families with asymptotic dimension at most $m$) then $\asdim(\cX) \leq m+n$.
\end{theorem}

\begin{remark}
In \cite[Proof of Theorem 4.1]{Guentner_Tessera_Yu2}, it is shown that the collection $\fD_n$ (used by Guentner, Tessera and Yu in the ordinal definition of FDC, \cite[Definition 2.2.1]{Guentner_Tessera_Yu2}) is contained in $\fA_{2^n-1}$. Theorem~\ref{thm:asdim-pullback} implies that $\fD_n$ is in fact contained in $\fA_n$.
\end{remark}

The proof of Theorem~\ref{thm:asdim-pullback}, which begins with the following lemma, is an adaptation of the Kolmogorov trick used in \cite{BDLM} to study the asymptotic dimension of metric spaces. 

We denote the number of elements of a finite set $F$ by $\#F$.

\begin{lemma}\label{lem:asdim-pullback}
Let $\cX$ be a metric family that is $n$-decomposable over $\fA_m$ and
let  $k$ be a non-negative integer.
Given $r >0$, for each $X \in \cX$
there is a decomposition  $X = X_0 \cup \cdots \cup X_{n+k}$  such that,
\begin{enumerate}
\item[(i)] for every $x \in X$,  $\#T_x \geq k+1$ where $T_x = \{ i ~|~ x \in X_i \} \subset \{ 0, 1, \ldots, n+k\}$;

\item[(ii)] for each $i$, $X_i=\bigsqcup_{\text{{\rm $r$-disjoint}}} \, X_{ij}$; and 

\item[(iii)] $\cX^\ast = \{ X_{ij} ~|~ \text{all $i,j$ and } X \in \cX \} \in \fA_m$.
\end{enumerate}
\end{lemma}

\begin{proof}

The proof proceeds  by induction on $k$ (our method is adapted from the proof of \cite[Theorem 2.4]{BDLM}).
The base case of the induction, $k=0$, is the given assertion that the metric family $\cX$ is $n$-decomposable over $\fA_m$.
Assume that the claim is valid for the integer $k$.
Let $r>0$.
For each $X\in \cX$, there is  a decomposition  $X = X_0 \cup \cdots \cup X_{n+k}$  such that,
for every $x \in X$,  $\#T_x \geq k+1$ and
$X_i=\bigsqcup_{\text{$3r$-disjoint}} \, X_{ij}$   and
\[
\cX^\ast = \{ X_{ij} ~|~ \text{all $i,j$ and } X \in \cX \} \in \fA_m.
\]
Let $X'_{ij}$ be the $r$-neighborhood of $X_{ij}$, that is,
\[
X'_{ij}  = \{ y ~|~ \text{there exists } x \in X_{ij} \text{ such that } d(x,y)< r\}.
\]
For  $0\leq i \leq n+k$,
let $\cU'_{i} = \{ X'_{ij}  ~|~ j\}$ and  $X'_i = \bigcup \,\cU_i'$.
Observe that  $\cU'_{i}$ is $r$-disjoint.
Let $\cU'_{n+k+1}$ be the collection of subspaces of the form
\[
X_{I,J} = X_{i_1 j_1} \cap \cdots \cap X_{i_{k+1} j_{k+1}}  ~\setminus ~ \bigcup_{i \notin \{i_1, \ldots, i_{k+1}\}} X'_i
\]
where $I = \{i_1, \ldots, i_{k+1}\}$ consists of $k+1$ distinct elements of  $\{0,\ldots,n+k\}$  and  \linebreak 
$J = \{j_1, \ldots, j_{k+1}\}$.
Let $I' = \{i'_1, \ldots, i'_{k+1}\}$ be another set of $k+1$ distinct elements of  $\{0,\ldots,n+k\}$.
If $I = I'$ then clearly any two distinct sets of the form $X_{I,J}$ and $X_{I',J'}$ are $r$-disjoint.
Suppose $I \neq I'$,  $a \in X_{I,J}$, $b \in X_{I',J'}$ and $d(a,b) < r$.
Since $I \neq I'$ there exists $i'_\ell \notin \{i_1, \ldots, i_{k+1}\}$.
Note that $b \in X_{i'_\ell}$.
However, by definition of $a\in X_{I,J}$, $a \notin X'_{i'_\ell}$. This implies that $d(a,b) \geq r$, a contradiction. 
Hence $X_{I,J}$ and $X_{I',J'}$ are $r$-disjoint.
This shows that $\cU'_{n+k+1}$ is $r$-disjoint.
Also note that $\cU'_{n+k+1}\in \fA_m$.

Let $X'_{n+k+1} = \bigcup \, \cU'_{n+k+1}$.
Observe that $X = X'_0 \cup \cdots \cup X'_{n+k+1}$.
For $x \in X$, let $S_x = \{ \ell ~|~  x \in X'_\ell \text{ and } 0 \leq \ell \leq n+k \}$ and $T'_x =  \{ \ell ~|~  x \in X'_\ell, \,\, 0 \leq \ell \leq n+k+1 \}$.
If $\#S_x \geq k+2$, then $\#T_x' \geq k+2$, so we are done.
Suppose $\#S_x = k+1$. 
Then for some $J$, $x \in X_{i_1j_1}\cap \cdots \cap X_{i_{k+1}j_{k+1}}$, where $i_l\in T_x\subset S_x$ and $j_l\in J$. If $x\notin X_{S_X,J}$, then $x \in X'_\ell$ for some $\ell \notin S_x$, a contradiction.
Hence, $x \in X_{S_X,J} \subset X'_{n+k+1}$ and so $\# T'_x \geq k + 2$.
This completes the induction step.
\end{proof}

\begin{proof}[Proof of Theorem~\ref{thm:asdim-pullback}]

\noindent
By Lemma \ref{lem:asdim-pullback},
given $r >0$, for each $X \in \cX$
there is a decomposition 
$X = X_0 \cup \cdots \cup X_{n+m}$  such that,
\begin{enumerate}
\item[(i)] for every $x \in X$,  $\#T_x \geq k+1$ where $T_x = \{ i ~|~ x \in X_i \} \subset \{ 0, 1, \ldots, n+k\}$;

\item[(ii)] for each $i$, $X_i=\bigsqcup_{\text{{\rm $r$-disjoint}}} \, X_{ij}$; and 

\item[(iii)] $\cX^\ast = \{ X_{ij} ~|~ \text{all $i,j$ and } X \in \cX \} \in \fA_m$.
\end{enumerate}
By \cite[Theorem 2.4]{BDLM},
there is an $(m,m+n+1)$-dimensional control function, $\cD$, for $\cX^\ast$, 
that is, for  $r>0$ as above there are covers $\cU^{ij}$ of $X_{ij} \in \cX^\ast$ such that,
\begin{enumerate}
\item[(1)]  $\cU^{ij} = \cU^{ij}_0 \cup \cdots \cup \cU^{ij}_{m +n}$,

\item[(2)] $\cU^{ij}_{\ell}$ is $r$-disjoint and  $\cD(r)$ bounded,

\item[(3)] For every $x \in X_{ij}$,  $\#S_x \geq n+1$ where
$
S_x = \{ \ell ~|~  \text{$x$ belongs to an element of $\cU^{ij}_\ell$} \}.
$
\end{enumerate}

Let $\cV_i = \bigcup_j \cU^{ij}_i$ for $i =0, \ldots, m+n$.  Note that $\cV_i$ is $r$-disjoint and $\cD(r)$-bounded.
For $x \in X \in \cX$, let 
$
S'_x = \{ \ell ~|~   \text{$x$ belongs to $\cU^{ij}_\ell$ for some $i,j$}    \}.
$
Recall that $\#T_x \geq m+1$, where $T_x = \{ i ~|~ x \in X_i \}$. 
Note that $S_x \subset S'_x$ and so $\#S'_x \geq n+1$.
Hence $S'_x \cap T_x \neq \emptyset$ because $S'_x$ and $T_x$ are subsets of $\{0, \ldots, m+n\}$, a set with $m+n+1$ elements.
Since $x \in X_{ij}$ implies that $S_x \subset S_x'$, it follows that for $\mu \in S'_x \cap T_x$, there exists $j$ and $U \in \cU^{\mu j}_\mu$ with $x \in U$.
This shows that $\cV =  \cV_0 \cup \cdots \cup  \cV_{m+n}$  is a cover of $\cX$.
It follows that $\asdim(\cX) \leq m+n$.
\end{proof}

We conclude this section by stating three alternative definitions (established in~\cite{Nicas-Rosenthal}) for a metric family $\cX$ to be $n$-decomposable over a collection of metric families $\fC$. 
Condition~(C) is a key technical tool needed for the proofs of Theorems~\ref{thmA-asdim},~\ref{thmAn} and~\ref{thmA}.

In what follows, let $\cX=\{X_\alpha \, | \, \alpha \in I\}$ be a metric family, where $I$ is a countable indexing set, and let $\fC$ be a collection of metric families. Let $n$ be a non-negative integer.

\begin{conditionA}\label{A}
	For every $d>0$, there exists a cover $\cV_\alpha$ of $X_\alpha$, for each $\alpha \in I$, such that:
	\begin{enumerate}	
	 	\item[(i)] the $d$-multiplicity of $\cV_\alpha$ is at most $n+1$ for every $\alpha \in I$; and
		\item[(ii)] $\bigcup_{\alpha\in I}\cV_\alpha$ is a metric family in $\fC$.
	\end{enumerate}
\end{conditionA}

\begin{conditionB}\label{B}
	For every $\lambda>0$, there exists a cover $\cU_\alpha$ of $X_\alpha$, for each $\alpha \in I$, such that:
	\begin{enumerate}	
	 	\item[(i)] the multiplicity of $\cU_\alpha$ is at most $n+1$ for every $\alpha \in I$;
		\item[(ii)] the Lebesgue number $L(\cU_\alpha)\geq \lambda$ for every $\alpha \in I$; and
		\item[(iii)] $\bigcup_{\alpha\in I}\cU_\alpha$ is a metric family in $\fC$.
	\end{enumerate}
\end{conditionB}

\begin{conditionC}\label{C}
	For every $\varepsilon>0$, there exists a uniform simplicial complex $K_\alpha$ and an $\varepsilon$-Lipschitz map $\varphi_\alpha:X_\alpha \to K_\alpha$, for each $\alpha \in I$, such that:
	\begin{enumerate}	
	 	\item[(i)] $\dim(K_\alpha)\leq n$ for every $\alpha \in I$; and
		\item[(ii)] $\bigcup_{\alpha\in I}\big\{ \varphi_\alpha^{-1}\big(\st(v)\big) \; \big| \; v \in K_\alpha^{(0)} \big\}$ is a metric family in~$\fC$.
	\end{enumerate}
\end{conditionC}
(We recall the notion of a uniform simplicial complex and of  the open star of  a vertex $v$, $\text{star}(v)$,  in Section \ref{sec:main-thm}.)

The following result is proved in~\cite[Propositions 9.18 and 9.20]{Nicas-Rosenthal}.

\begin{proposition}\label{prop:equiv}
	Let $\cX$ be a metric family and $\fC$ be a collection of metric families that satisfies Coarse Permanence. Then Conditions~(A), (B) and (C) are each equivalent to Definition~\ref{def:n-decompose}.\qed	
\end{proposition}


\section{Decomposition of a Group Over a Collection of Metric Families}\label{sec:main-thm}

In \cite{Bartels1}, Bartels reformulated, in geometric group theoretic terms,
the conditions that he used with L\"{u}ck and Reich in~\cite{BLR08} to prove the Farrell-Jones Conjecture for word hyperbolic groups.
Willet and Yu observed (see \cite[Remark 1.3.5]{Bartels1}) that a group satisfying these conditions must have finite asymptotic dimension.
In this section, we approach this fact from the viewpoint of metric families in Theorem~\ref{thmAn} allowing us to give a generalization that applies to more general coarse geometric notions (Theorem~\ref{thmA}).
The conditions used by Bartels in \cite{Bartels1}, which are of a technical nature, evolved into his notion of finitely $\cF$-amenable actions in~\cite{Bartels17}.
We formulate our results using this language in Section~\ref{sec:amenable}.

Let $G$ be a group with a left invariant metric $\rho$.
Recall that
the corresponding {\it norm} on $G$ is given by ${\| g \|}_\rho = \rho(e,g)$ for $g \in G$;
moreover, $\rho$ can be recovered from ${\| \cdot  \|}_\rho$ via $\rho(g, h) ={\|g^{-1}h\|}_\rho$.
The following notion of ``$G$-equivariant up to $\varepsilon$'' generalizes the corresponding definition given by Bartels in \cite{Bartels1} (see Proposition \ref{prop:word_length}).

\begin{definition}\label{def:equivariant_up_to}
Let $G$ be a group with a left invariant metric $\rho$.
Let $\varepsilon \geq 0$.
A map $f \colon X \to Y$ between $G$-spaces, where $Y$ is equipped with a left invariant metric $d$,
is {\it $G$-equivariant up to $\varepsilon$}  if
$d\big(f(g x),g f(x)\big)\leq  \varepsilon {\| g \|}_\rho$ for every $g \in G$ and every $x\in X$.
\end{definition}

A map that is $G$-equivariant up to $\varepsilon$ gives rise to a family of $\varepsilon$-Lipschitz maps as follows.

\begin{lemma}\label{lem:bartels-lipschitz}
Let $G$ be a group with a left invariant metric $\rho$.
Let $X$ and $Y$ be left $G$-spaces, and let $d$ be a left-invariant metric on $Y$. 
Let $f \colon X \to Y$ be a map that is $G$-equivariant up to $\varepsilon$.
Then, for any $x \in X$, the map $\varphi_x \colon G \to Y$ defined by $\varphi_x(g)=g f(g^{-1} x)$ is $\varepsilon$-Lipschitz.
\end{lemma}

\begin{proof}
We have,
\begin{eqnarray*}
	d\big(\varphi_x(g),\varphi_x(h)\big) & = & d\big(g  f(g^{-1} x),h  f(h^{-1} x)\big) \\
	  & = & d\big(f(g^{-1} x),(g^{-1}h) f(h^{-1}  x)\big) \\
	  & = & d\big(f((g^{-1}h)(h^{-1} x)),(g^{-1}h) f(h^{-1}  x)\big)  \\
	  & \leq & \varepsilon \, {\|g^{-1}h\|}_\rho = \varepsilon \,  \rho(g,h),
\end{eqnarray*}
that is,  $\varphi_x$ is $\varepsilon$-Lipschitz.
\end{proof}

Let $G$ be a countable group with a possibly infinite symmetric generating set $S$.
A  {\it weight function} for $S$ is  a function  $W \colon S  \rightarrow \R$ that is
positive,  $W(s) = W(s^{-1})$ for $s \in S$, and  proper (the pre-image of a bounded 
interval is a finite subset of $S$).
The weight function $W$ defines a proper norm on $G$ by
\[
\| g \|_W = \inf \left\{ \sum^n_{j=1} W(s_j) ~\Big|~ g = s_1 s_2 \cdots s_n,  \,\, s_j \in S \right\},
\]
see
\cite[Proposition 1.3]{DranishnikovSmith}.
Note that the requirement that  $W$ is proper implies that the infimum in the above definition is attained.
The corresponding left-invariant {\it weighted word length metric} on $G$ is given by $d_W(g,h) = \|g^{-1}h \|_W$.
In the case $S$ is finite, implying $G$ is finitely generated,
we can take $W=1$  (the constant function with value $1$) obtaining the
familiar
{\it word length metric} on $G$.

For a  countable group $G$ with a symmetric generating set $S$ and weight function $W$, 
an equivalent characterization of ``$G$-equivariant up to $\varepsilon$'',
recovering the original definition of Bartels  \cite{Bartels1} in the case $S$ is finite and $W=1$,
is given as follows.

\begin{proposition}\label{prop:word_length}
Let $G$ be a  countable group with a symmetric generating set $S$ and weight function $W$.
A map $f \colon X \to Y$ as in Definition \ref{def:equivariant_up_to} is $G$-equivariant up to $\varepsilon$ if and only if
for all $s \in S$ and $x \in X$, $d\big(f(s x),s f(x)\big)\leq  \varepsilon \, W(s)$.
\end{proposition}

\begin{proof}
Assume 
$d\big(f(g x),g f(x)\big)\leq \varepsilon \, \|g\|_W$ for all $g \in G$ and $x \in X$.
If  $s \in S$  then $\|s\|_W \leq W(s)$ and so $d\big(f(s x),s f(x)\big)\leq \varepsilon \, W(s) $.

Assume $d\big(f(s x),s f(x)\big)\leq \varepsilon \, W(s)$ for all $s \in S$ and $x \in X$.
Let $s\in S$, $h \in G$ and $g = sh$.
We have that
\begin{eqnarray*}
	d\big(f(g x),g f(x)\big) & = & d\big(f(s(hx)),s(h f(x))\big) \\
	  & \leq & d\big(f(s(hx)),s f(h x)\big) + d\big(s f(h  x),s(h f(x))\big) \\
	  & = & d\big(f(s(hx)),s f(h x)\big) + d\big(f(h  x), h f(x)\big)  \\
	  & \leq & \varepsilon \, W(s) +  d\big(f(h  x), h f(x)\big).
\end{eqnarray*}
If $g = s_1 s_2 \cdots s_n$, $s_j \in S$, then repeated application of the above inequality yields
\[
d\big(f(g x),g f(x)\big)  \leq  \varepsilon \left( \sum^n_{j=1} W(s_j)\right).
\]
Taking the infimum over all expressions of $g$ as a product of elements $S$ we obtain $d\big(f(g x),g f(x)\big)\leq \varepsilon \, \|g\|_W$,
that is, $f$ is $G$-equivariant up to $\varepsilon$.
\end{proof}

We briefly summarize some facts that we need about uniform simplicial complexes.

Given a set $S$ and a real valued function $f \colon S \to \R$,
the {\it support} of $f$ is the set $\supp(f) =\{ s \in S ~|~ f(s) \neq 0\}$.
The {\it real vector space with basis $S$} is the set $\R^S =\{ f \colon S \to \R ~|~ \text{$\supp(f)$ is finite}\}$ together
with the usual operations given by pointwise addition and scalar multiplication of functions.
The {\it $\ell^1$-norm} on $\R^S$, denoted  $\| \cdot \|_1$, is given by
$
\| f \|_1 =  \sum_{s \in S} |f(s)| 
$
and the corresponding {\it $\ell^1$-metric} on $\R^S$ is given by $d^1(f,g) = \| f  -g \|_1$.

The {\it simplex with vertex set $S$}, denoted $\Delta(S)$, is the subset of $\R^S$ given by
\[
\Delta(S)=\left\{ f \in \R^S ~\Big|~ f \geq 0 \text{ and }  \sum_{s \in S} f(s) = 1\right\}.
\]
The topology on $\Delta(S)$ induced by the $\ell^1$-metric on $\R^S$ is called the {\it strong topology}.

Recall that a {\it simplicial complex} consists of a vertex set $K^{(0)}$ together with a collection $K$ of non-empty finite subsets (``simplices'') of $K^{(0)}$ such that:
\begin{itemize}
\item For every $x \in K^{(0)}$, $\{x\} \in K$,

\item if $\sigma \in K$ and $\tau \subset  \sigma$ and $\tau$ is non-empty then $\tau \in K$.
\end{itemize}
For brevity, we sometimes write $K$ for  $(K^{(0)}, K)$.

The {\it geometric realization} of a simplicial complex $(K^{(0)}, K)$, denoted by $|K|$, is the subset of $\Delta(K^{(0)})$ given by
\[
|K| ~=~ \left\{ f \in \Delta\big(K^{(0)}\big) ~\big|~ \supp(f) \in K \right\}.
\]
Note that the metric $d^1$ on $\Delta\big(K^{(0)}\big)$ restricts to  $|K|$.   A metric space of the form $(|K|, d^1)$ is called a {\it uniform simplicial complex.}

A vertex $v \in  K^{(0)}$ can be viewed as an element of $\Delta\big(K^{(0)}\big)$, also denoted by $v$ (to avoid an excess of notation), specified
by $v(v) = 1$ and $v(u) =0$ for $u \neq v$.
Any element $x\in |K|$ can be uniquely written as $x=\sum_{v\in K^{(0)}}x_v \, v$,
where $0 \leq x_v  \leq 1 $ and $x_v=0$ for all but finitely many $v\in K^{(0)}$, and $\sum_{v\in K^{(0)}}x_v=1$.
The number $x_v$ is called the {\it barycentric coordinate} of $x$ corresponding to the vertex $v$.
The {\it open star} of a vertex $v\in K^{(0)}$ is the set $\st(v)=\left\{ x\in |K| \;\big|\; x_v\neq 0 \right\}$.

Let $G$ be a group and $(K^{(0)}, K)$ a simplicial complex.  A {\it simplicial (left) $G$-action} on $K$ is a (left) $G$-action on the vertex set $K^{(0)}$ 
such that if $\{v_0, \ldots, v_n\}$ is an $n$-simplex of $K$ then for any $g \in G$, $\{g \,v_0, \ldots, g\,  v_n\}$ is also an $n$-simplex of $K$.
A simplicial $G$-action on $K$ yields a left $G$-action on $|K|$ via the formula  $(g  f)(s) = f(g^{-1}  s)$ for $g \in G$ and $s \in K^{(0)}$.  
For all $f, h \in |K|$ we have $d^1\big(g  f, g  h\big) = d^1\big(f, h\big)$, that is, $G$ acts by
isometries on $|K|$.

There is another useful topology on the geometric realization of a simplicial complex.
The {\it weak topology}, also known as the {\it Whitehead topology},
on the underlying set of $|K|$ is characterized as follows: a subset
$A \subset |K|$ is closed in the weak topology if and only if for every simplex $\sigma \in K$, the set $A \cap |\sigma|$ is closed in $|\sigma|$.
We denote the corresponding topological space by $|K|_w$.  
With this topology, $|K|_w$ is a CW complex with $n$-skeleton  $(|K|_w)^{(n)} =  \bigcup \{ |\sigma| ~|~  \sigma \in K \text{ and } \dim(\sigma) \leq n\}$.
The weak topology is finer than the strong topology, that is, the identity map ${\widetilde \id} \colon |K|_w \to |K|$ is continuous.
A simplicial complex $K$ is {\it locally finite} if each vertex of $K$ belongs to only finitely many simplicies of $K$.

\begin{proposition}[{\cite[Proposition 3.3.14]{FP}}] \label{locallyfiniteprop}
Let $K$ be a simplicial complex. The weak topology on the underlying set of $|K|$ coincides with the strong topology if and only if
$K$ is locally finite.\qed	
\end{proposition}

Although ${\widetilde \id} \colon |K|_w \to |K|$ is not a homeomorphism if $|K|$ is not locally finite, it is always a homotopy equivalence by \cite[\S16, Theorem 1]{Dowker}.
The following proposition is a consequence of Dowker's  theory of  metric complexes (see \cite[\S14 and (15.2)]{Dowker}
and note that $|K|$ with the $\ell^1$-metric is a metric complex in the sense of Dowker).

\begin{proposition} 
\label{DowkerProp}
Let $\varepsilon > 0$ be given. 
There exists a continuous map $h\colon |K| \to |K|_w$ such that ${\tilde h} ={\widetilde \id} \circ h$ is $\varepsilon$-homotopic,
with respect to the $\ell^1$-metric on $|K|$,
to the identity map $\id_{|K|} \colon |K| \to |K|$.\qed
\end{proposition}

\begin{lemma}
\label{Whiteheadtopologylemma}
Let $G$ be a countable group with a proper left invariant metric $\rho$ and let $C= \min\{ {\|g\|}_\rho ~|~  g \neq 1 \}$.
Let $\varepsilon > 0$.
Let $X$ be a $G$-space, $|K|$ a uniform simplicial complex equipped with a simplicial $G$-action, and  $f \colon X \to |K|$ a map
that is $G$-equivariant up to $\varepsilon$ (Definition \ref{def:equivariant_up_to}).
Then $f$ is $\varepsilon$-homotopic to a map ${\tilde f}\colon X \to |K|$ that factors through ${\widetilde \id} \colon |K|_w \to |K|$ and satisfies
$d^1\big({\tilde f}(g x), g {\tilde f}(x)\big)\leq  \varepsilon (1  +  2/C)  {\| g \|}_\rho$  for all $g \in G$ and all $x \in X$.
\end{lemma}

\begin{proof}
By Proposition \ref{DowkerProp},
there is a map $h\colon |K| \to |K|_w$ such that ${\tilde h} ={\widetilde \id} \circ h$ is \hbox{$\varepsilon$-homotopic}
to $\id_{|K|} $. 
Let $H\colon |K| \times [0,1] \to |K|$ be a homotopy with $H_0=\id_{|K|}$, $H_1={\tilde h}$  and $d^1\big(H_t(y), y\big) \leq \varepsilon$  for all $t \in [0,1]$ and $y \in |K|$.
Let ${\tilde f}_t = H_t \circ f$ and let ${\tilde f} = {\tilde f}_1 = {\tilde h} \circ f$ yielding an $\varepsilon$-homotopy from $f$ to ${\tilde f}$.
For  all $t \in [0,1]$, $g \in G$, $g \neq 1$, and $x \in X$ we have
\begin{eqnarray*}
d^1\big({\tilde f}_t(g  x),g  {\tilde f}_t(x)\big) &=& d^1\big(H_t(f(g  x)),g  H_t(f(x))\big) \\
                                &\leq &  d^1\big(H_t(f(g  x)),  f(g  x)\big) + d^1\big(f(g  x),g  H_t(f(x))\big)\\
                                &\leq & \varepsilon + d^1\big(f(g  x),g  H_t(f(x))\big)\\
                                &\leq & \varepsilon + d^1\big(f(g  x),g  f(x) \big) + d^1\big(g  f(x), g  H_t(f(x))\big) \\
                                &= &     \varepsilon + d^1\big(f(g  x),g  f(x) \big) + d^1\big(f(x),H_t(f(x))\big)  \\
                                &\leq &  2 \varepsilon + d^1\big(f(g  x),g  f(x) \big)  ~ \leq~ 2 \varepsilon + \varepsilon {\| g \|}_\rho  ~\leq~ \varepsilon (1  +  2/C)  {\| g \|}_\rho.
\end{eqnarray*}
Hence ${\tilde f}$ satisfies the conclusion of the lemma.
\end{proof}

\begin{theorem}\label{thmA-asdim}
Let $G$ be a  countable group with a proper left invariant metric.
Let $k$ and $n$ be non-negative integers. Assume that for every $\varepsilon>0$ there is a compact $G$-space $X$, a uniform simplicial complex $E$ equipped with a simplicial $G$-action, and a map $f:X \to E$ such that
	\begin{enumerate}
		\item[(i)] $\dim(E)\leq n$;
		\item[(ii)] $f$ is $G$-equivariant up to $\varepsilon$ (Definition \ref{def:equivariant_up_to}),
		\item[(iii)] for each vertex $v \in E$, $\asdim(G_v)\leq k$, where $G_v$ is the stabilizer subgroup of $v$ and is viewed as a metric subspace of $G$.		
	\end{enumerate}
Then $\asdim(G)\leq n+k$.
\end{theorem}

\begin{proof}
We will show that $G$ is $n$-decomposable over $\fA_k$, the collection of metric families with asymptotic dimension at most $k$. Then, by Theorem~\ref{thm:asdim-pullback}, $\asdim(G)\leq n+k$.

We proceed by showing that $G$ satisfies Condition~(C) with respect to $n$ and $\fA_k$. 
Let $\varepsilon>0$ be given.
Then there is a compact $G$-space $X$, a uniform simplicial complex $E$ equipped with a simplicial $G$-action, and a map $f:X \to E$ that satisfy assumptions (i), (ii) and~(iii). 
By Lemma \ref{Whiteheadtopologylemma},  we can assume that $f$ factors through 
the identity map ${\widetilde \id} \colon E_w \to E$, where $E_w$ denotes the underlying set of $E$ topologized with the weak topology determined by the collection of closed simplices of $E$
(note that the weak topology and the metric topology on $E$ need not coincide, see Proposition \ref{locallyfiniteprop}).
The space $E_w$ is a CW-complex whose $n$-skeleton is the union of  all the closed simplices of $E$ of dimension at most $n$.
Since $X$ is compact and $f$ is continuous (as a map into $E_w$), $f(X)$ is a compact subset of $E_w$ and, thus, is contained in the union of finitely many simplices of $E$.

Let $E_f$ be a finite subcomplex of $E$ with $f(X) \subset E_f$, and let $\{ v_1,\dots, v_m \}$ be the vertex set of $E_f$. Let $G_{v_i}=\big\{ g\in G \;\big|\; g v_i=v_i \big\}$ be the stabilizer of $v_i$. For each vertex $w$ of $E$, let $Q_i(w)=\big\{ g\in G \;\big|\; g v_i=w \big\}$. If $Q_i(w)$ is non-empty, choose $g_{w,i} \in Q_i(w)$. Then $Q_i(w)=g_{w,i}G_{v_i}$. 

	Fix $x\in X$ and define $\varphi_x: G \to E$ by $\varphi_x(g)=g f(g^{-1} x)$; it is $\varepsilon$-Lipschitz by Lemma~\ref{lem:bartels-lipschitz}. If $v$ is a vertex of $E$, then
\renewcommand{\arraystretch}{2}
\[\begin{array}{rcl}
	\varphi_x^{-1}\big(\{ v \} \big) & = & \big\{g \in G \; \big|\; \varphi_x(g)=v \big\} \\
	  & = & \big\{g \in G \; \big|\; g f(g^{-1} x)=v \big\}  \\
	  & = & \big\{g \in G \; \big|\; f(g^{-1} x)=g^{-1} v \big\}  \\
	  & \subset & \big\{g \in G \; \big|\; g^{-1} v \in \{ v_1,\dots, v_m \} \big\}=\displaystyle\bigcup_{i=1}^m \big\{g \in G \; \big|\; g^{-1} v=v_i \big\},
\end{array}\]
since $f(X) \subset E_f$. Thus,
\[ \varphi_x^{-1}\big(\{ v \} \big) \subset \displaystyle\bigcup_{i=1}^m \big\{g \in G \; \big|\; g^{-1} v=v_i \big\} = \bigcup_{i=1}^m Q_i(v) = \bigcup_{i=1}^m g_{v,i}G_{v_i}. \]
More generally, we have
\[ \varphi_x^{-1}\big(\st(v)\big) \subset \bigcup_{i=1}^m g_{v,i}G_{v_i}, \] 
because
\[ g\in \varphi_x^{-1}\big(\st(v)\big) \;\; \Leftrightarrow \;\; \big(\varphi_x(g)\big)_v \neq 0 \;\; \Leftrightarrow \;\; f(g^{-1} x)_{g^{-1} v} \neq 0 \;\; \Leftrightarrow \;\; g^{-1} v \in \{ v_1,\dots, v_m \}. \]
(Recall $\big(\varphi_x(g)\big)_v$ and $f(g^{-1} x)_{g^{-1} v}$ denote barycentric coordinates.)
It follows that \newline
\[
\big\{ \varphi^{-1}_x\big(\st(v)\big) \; \big| \; v \in E^{(0)} \big\} \subset \Big\{ \bigcup_{i=1}^m g_{v,i}G_{v_i} \;\Big|\; v \in E^{(0)} \Big\}.
\]
By Corollary~\ref{thm:asdim-cosets-family}, $\asdim\big(\big\{ \bigcup_{i=1}^m g_{v,i}G_{v_i} \;\big|\; v \in E^{(0)} \big\} \big)\leq k$.
Thus,
\[
\asdim\big(\big\{ \varphi^{-1}_x\big(\st(v)\big) \; \big| \; v \in E^{(0)} \big\}\big) \leq k.
\]
Therefore, $G$ is $n$-decomposable over $\fA_k$.	
\end{proof}

\begin{remark}
Let $\rho$ and $\rho'$ be proper left invariant metrics on $G$. 
It was pointed out to us by the referee that 
if the conditions of the above theorem are satisfied for one of the metrics, then they are satisfied for the other by arguing as follows.
Assume that $\rho$ satisfies the conditions of the theorem.  
Let $\varepsilon > 0$.
Since the open ${\rho'}$-ball of radius $2/ \varepsilon$ centered at the identity is finite,
there is a $\delta >0$ such that $\delta \,  \| g \|_{\rho} \leq  \| g \|_{\rho'}$ for all $g \in G$ with $\|g\|_{\rho'} < 2/ \varepsilon$.
Let $f \colon X \rightarrow E$ be $G$-equivariant up to $\varepsilon \delta$ with respect to the metric $\rho$.
Then for $\| g \|_{\rho'} < 2/ \varepsilon$ and $x \in X$, we have
\[
d^1(f(gx), g f(x)) ~\leq~ \varepsilon \delta \, \| g \|_{\rho} ~\leq~  \varepsilon \, \| g \|_{\rho'}.
\]
For $\| g \|_{\rho'} \geq 2/ \varepsilon$ and $x \in X$, we have
\[
\varepsilon \, \| g \|_{\rho'}  ~\geq~  2 ~\geq~ \diam(E) ~\geq~  d^1(f(gx), g f(x)).
\]
Hence, $f$ is $G$-equivariant up to $\varepsilon$ with respect to the metric $\rho'$.
\end{remark}

Guentner, Willet and Yu showed that if the stabilizer subgroups, $G_v$, are finite for every vertex $v$ in $E$, then the action of $G$ on $X$ has finite {\it dynamic asymptotic dimension}~\cite[Theorem 4.11]{GWY}.

We wish to generalize Theorem~\ref{thmA-asdim} to allow for isotropy groups that are contained in a collection of metric families with sufficiently nice properties.

\begin{theorem}\label{thmAn}
	Let $\fC$ be a collection of metric families that satisfies Coarse Permanence~(\ref{def:closed-under-embeddings}), Finite Amalgamation Permanence~(\ref{def:amalgam}), and Finite Union Permanence~(\ref{def:finunion}). Let $G$ be a  countable group with a proper left invariant metric, and let $n$ be a non-negative integer. Assume that for every $\varepsilon>0$ there is a compact $G$-space $X$, a uniform simplicial complex $E$ equipped with a simplicial $G$-action, and a map $f \colon X \to E$ such that
	\begin{enumerate}
		\item[(i)] $\dim(E)\leq n$;
		\item[(ii)] $f$ is $G$-equivariant up to $\varepsilon$ (Definition \ref{def:equivariant_up_to});
		\item[(iii)] for each vertex $v \in E$, the stabilizer subgroup $G_v=\big\{ g\in G \;\big|\; g v=v \big\}$, considered as a metric subspace of $G$, is in $\fC$.		
	\end{enumerate}
Then $G$ is $n$-decomposable over $\fC$. In particular, if $\fC$ is also stable under weak decomposition, then $G$ is in $\fC$.
\end{theorem}

\begin{proof}
	By Proposition~\ref{prop:equiv}, the result will follow from showing that $G$ satisfies Condition~(C) with respect to $n$ and $\fC$. 

	Let $\varepsilon>0$ be given. Then there is a compact $G$-space $X$, a uniform simplicial complex $E$ equipped with a simplicial $G$-action, and a map $f:X \to E$ that satisfy assumptions (i), (ii) and~(iii). As in the proof of Theorem~\ref{thmA-asdim}, we can assume, by Lemma~\ref{Whiteheadtopologylemma}, that there is a finite subcomplex $E_f$ of $E$ with $f(X) \subset E_f$. Let $\{ v_1,\dots, v_m \}$ be the vertex set of $E_f$, and let $G_{v_i}=\big\{ g\in G \;\big|\; g v_i=v_i \big\}$ be the stabilizer of $v_i$. Then, as established in the proof of Theorem~\ref{thmA-asdim}, $\big\{ \varphi^{-1}_x\big(\st(v)\big) \; \big| \; v \in E^{(0)} \big\} \subset \big\{ \bigcup_{i=1}^m g_{v,i}G_{v_i} \;\big|\; v \in E^{(0)} \big\}$. 
	
	Since the metric on $G$ is left-invariant, each $g_{v,i}G_{v_i}$ is isometric to $G_{v_i}$. Therefore, $\big\{g_{v,i}G_{v_i} \;\big|\; 1 \leq i \leq m, \, v \in E^{(0)} \big\}$ is coarsely equivalent to the metric family $\big\{G_{v_i} \;\big|\; 1 \leq i \leq m \big\}$, which is in $\fC$ by Finite Amalgamation Permanence. Thus, $\big\{g_{v,i}G_{v_i} \;\big|\; 1 \leq i \leq m, \, v \in E^{(0)} \big\}$ is in $\fC$ by Coarse Permanence, and so $\big\{ \bigcup_{i=1}^m g_{v,i}G_{v_i} \;\big|\; v \in E^{(0)} \big\}$ is in $\fC$ by Finite Union Permanence. Since inclusions are a special case of Coarse Permanence, $\big\{ \varphi^{-1}_x\big(\st(v)\big) \; \big| \; v \in E^{(0)} \big\}$ is in $\fC$.
Thus, $G$ satisfies Condition~(C) with respect to $n$ and $\fC$, as desired.
\end{proof}

Theorem~\ref{thmA-asdim} is a special case of Theorem~\ref{thmAn} since the collection $\fA_k$ of metric families with asymptotic dimension less than or equal to $k$ satisfies Coarse Permanence, Finite Amalgamation Permanence, and Finite Union Permanence.

Recall that if a collection of metric families, $\fC$, satisfies Coarse Permanence and Finite Amalgamation Permanence and is stable under weak decomposition, then it will automatically satisfy Finite Union Permanence (Lemma~\ref{lem:weak-decomp-finite-union}). Thus, in Theorem~\ref{thmAn}, if we replace the assumption that $\fC$ satisfies Finite Union Permanence with the assumption that it is stable under weak decomposition, then we can still conclude that $G$ is in $\fC$. We state this result formally below.

\begin{theorem}\label{thmA}
	Let $\fC$ be a collection of metric families that satisfies Coarse Permanence, Finite Amalgamation Permanence, and is stable under weak decomposition. Let $G$ be a  countable group with a proper left invariant metric, and let $n$ be a non-negative integer. Assume that for every $\varepsilon>0$ there is a compact $G$-space $X$, a uniform simplicial complex $E$ equipped with a simplicial $G$-action, and a map $f \colon X \to E$ such that
	\begin{enumerate}
		\item[(i)] $\dim(E)\leq n$;
		\item[(ii)] $f$ is $G$-equivariant up to $\varepsilon$ (Definition \ref{def:equivariant_up_to});
		\item[(iii)] for each vertex $v \in E$, the stabilizer subgroup $G_v=\big\{ g\in G \;\big|\; g v=v \big\}$,
		considered as a metric subspace of $G$, is in $\fC$.		
	\end{enumerate}
Then $G$ is in $\fC$.\qed	
\end{theorem}

\begin{remark}\label{rem:collections}
Among the collections of metric families considered in this paper (that is, $\fB$, $\fA$, $\fA_k$, $\fD$, $\wD$, $\fE$, and $\fH$), all of them satisfy Coarse Permanence, Finite Amalgamation Permanence, and Finite Union Permanence (see, for example, \cite{Guentner}). The collections $\wD$, $\fE$, and $\fH$ are known to be stable under weak decomposition (\cite[Proof of Theorem 4.3]{Guentner_Tessera_Yu2}, \cite[Theorem 9.23]{Nicas-Rosenthal}). It is unknown if $\fD$ is stable under weak decomposition.
\end{remark}


\section{Finitely $\cF$-amenable actions}\label{sec:amenable}

In his work on relatively hyperbolic groups and the Farrell-Jones Conjecture, Bartels introduced the notion of a finitely $\cF$-amenable action~\cite{Bartels17}, where $\cF$ is a family of subgroups of a given group that is closed under conjugation and taking subgroups. Such actions provide examples to which the results of Section~\ref{sec:main-thm} can be applied, for example to relatively hyperbolic groups (see Theorem~\ref{thm:rel_hyp} below).

\begin{definition}\label{def:F-cover}
Let $X$ be a $G$-space and $\cF$ be a family of subgroups of $G$.
\begin{enumerate}
	\item An open set $U$ in $X$ is an {\it $\cF$-subset} if there is an $F\in \cF$ such that $gU=U$ for every $g\in F$ and $gU\cap U=\emptyset$ for every $g \notin F$.
	\item An open cover $\cU$ of $X$ is {\it $G$-invariant} if $gU\in \cU$ for all $g\in G$ and all $U \in \cU$.
	\item A $G$-invariant cover $\cU$ of $X$ is an {\it $\cF$-cover} if all of the members of $\cU$ are $\cF$-subsets.
\end{enumerate}
\end{definition}

\begin{definition}\cite[Definition 0.1]{Bartels17}\label{def:NF-amenable}
Let $\cF$ be a family of subgroups of $G$ and let $N$ be a non-negative integer.
A $G$-action on a space $X$ is {\it $N$-$\cF$-amenable} if for any finite subset $S$ of $G$ there exists an open $\cF$-cover $\cU$ of $G\times X$ (equipped with the diagonal $G$-action) such that:
\begin{enumerate}
	\item the dimension of $\cU$ is at most $N$; and 
	\item for all $x\in X$ there is a $U \in \cU$ with $S \times \{x\} \subseteq U$.
\end{enumerate}
A $G$-action is called {\it finitely $\cF$-amenable} if it is $N$-$\cF$-amenable for some $N$.
\end{definition}

\begin{proposition}\label{prop:NF-amenable}
Let $G$ be a countable group and $\cF$ be a family of subgroups of $G$. Assume that there exists an $N$-$\cF$-amenable action of $G$ on $X$, where $X$ is compact and metrizable. Then, given any proper left invariant metric on $G$, for every $\varepsilon>0$ there exists a uniform simplicial complex $E$ equipped with a simplicial $G$-action and a map $f:X \to E$ such that:
\begin{enumerate}
	\item[(i)] $\dim(E)\leq N$;
	\item[(ii)] $f$ is $G$-equivariant up to $\varepsilon$ (Definition \ref{def:equivariant_up_to}),
	\item[(iii)] the stabilizer subgroup of each vertex in $E$ is an element of $\cF$.		
\end{enumerate}
\end{proposition}

\begin{proof}
The proposition follows from \cite[Remarks 0.2--0.4]{Bartels17}, but we include a proof here for the reader's convenience.

Let $d_G$ be a proper left invariant metric on $G$. (Recall that every countable group admits one.) Thus, every ball of finite radius in $G$ with respect to $d_G$ contains finitely many elements. Let $\varepsilon>0$ be given and set $R=\frac{(2N+2)(2N+3)}{\varepsilon}$. Since the action of $G$ on $X$ is $N$-$\cF$-amenable, there exists an open $\cF$-cover $\cU$ of $G \times X$ of dimension at most $N$ such that for each $x\in X$ there is a $U\in \cU$ with $B_{d_G}(e;R) \times \{x\} \subseteq U$, where $B_{d_G}(e;R)$ is the open ball of radius $R$ in $G$ around the identity of $G$.  Let $E={\rm Nerve}(\cU)$ equipped with the uniform metric. Then $\dim(E)=\dim(\cU) \leq N$. It follows from the definition of an $\cF$-cover that the stabilizer subgroup of a vertex in $E$ is an element of the family $\cF$.
	
Since $X$ is compact and metrizable, by \cite[Proposition 4.3 and Lemma 5.1]{BLR08} there is a metric $d$ on $G \times X$ that is $G$-invariant with respect to the diagonal action, satisfies $d\big((g,x),(h,x)\big)=d_G(g,h)$ for every $g,h \in G$ and $x \in X$, and has the property that for every $(g,x)\in G \times X$ there is a $U_{(g,x)}\in \cU$ such that $B_{d}\big((g,x);R\big) \subseteq U_{(g,x)}$, where $B_{d}\big((g,x);R\big)$ is the open ball of radius $R$ around $(g,x)$ with respect to $d$.

We now define a map $f:X \to E$ that satisfies item (ii). Recall the following standard construction. For each $U\in \cU$, define $\psi_U: G \times X \to [0,1]$ by
	\[ \psi_U(y)=\dfrac{d(y,U^c)}{\sum_{V \in \;\cU}d(y,V^c)} \]
where $U^c$ is the complement of $U$ in $G \times X$. Define $\psi:G \times X \to E$ to be the map
	\[ \psi(y)=\sum_{U \in \; \cU}\psi_U(y) \cdot [U] \]
where $[U]$ denotes the vertex of $E$ corresponding to $U$.
Note that $\psi$ is $G$-equivariant: Since $d(gy,U^c)=d(y,g^{-1}U^c)$, it follows that $\psi_U(gy)=\psi_{g^{-1}U}(y)$ and so
\[ g\psi(y)=\sum_{U \in \; \cU}\psi_U(y) \cdot [gU]=\sum_{U \in \; \cU}\psi_{g^{-1}U}(y) \cdot [U]=\sum_{U \in \; \cU}\psi_U(gy) \cdot [U]=\psi(gy). \] 
Let  $f:X \to E$ be the map defined by $f(x)=\psi(e,x)$.

For every $y,y' \in G \times X$ and $U \in \cU$, the triangle inequality implies
	\[ \big| d(y,U^c) - d(y',U^c)\big| \leq d(y,y'). \]
Therefore, {\small
\renewcommand{\arraystretch}{3}
\[\begin{array}{rcl}
	\big|\psi_U(y)-\psi_U(y')\big| & = & \left| \dfrac{d(y,U^c)}{\sum_{V \in \;\cU}d(y,V^c)} - \dfrac{d(y',U^c)}{\sum_{V \in \;\cU}d(y',V^c)}  \right| \\
	& \leq & \dfrac{|d(y,U^c) - d(y',U^c)|}{\sum_{V \in \;\cU}d(y,V^c)} + \left| \dfrac{d(y',U^c)}{\sum_{V \in \;\cU}d(y,V^c)} - \dfrac{d(y',U^c)}{\sum_{V \in \;\cU}d(y',V^c)} \right| 
\end{array} \] }

\noindent
which is less than or equal to

\noindent
{\small
\renewcommand{\arraystretch}{2}
	\[ \dfrac{d(y,y')}{\sum_{V \in \;\cU}d(y,V^c)} +  \dfrac{d(y',U^c)}{\Big(\sum_{V \in \;\cU}d(y,V^c)\Big) \Big(\sum_{V \in \;\cU}d(y',V^c)\Big)}\cdot \sum_{V \in \;\cU} \big|d(y,V^c)-d(y',V^c)\big|\]}

\noindent
which is less than or equal to

\noindent
\renewcommand{\arraystretch}{2}
	\[ \dfrac{1}{\sum_{V \in \;\cU}d(y,V^c)}\Big(d(y,y') +  \sum_{V \in \;\cU} \big|d(y,V^c)-d(y',V^c)\big| \Big).\]

Let $g,h \in G$ and $x \in X$ be given. Since $B_{d}\big((g,x);R\big) \subseteq U_{(g,x)}$, it follows that  
\[ \sum_{V \in \;\cU}d\big((g,x),V^c\big) \geq d\big((g,x), (U_{(g,x)})^c \big)\geq R. \]
Thus,
\renewcommand{\arraystretch}{2}
\[\begin{array}{rcl}
	\big|\psi_U(g,x)-\psi_U(h,x)\big| & \leq & \frac{1}{R}\, \Big(d\big((g,x),(h,x)\big) + 2(N+1)\, d\big((g,x),(h,x)\big)\Big)  \\
	& = & \frac{1}{R}\, (2N+3)\,  d_G(g,h)
\end{array}\]

\noindent
and so, 
$$ d^1\big(\psi(g,x),\psi(h,x)\big) \, = \sum_{U \in \;\cU} \big|\psi_U(g,x)-\psi_U(h,x)\big| \,\leq\, \frac{2(N+1)(2N+3)}{R}\, d_G(g,h).
$$

\noindent
Hence, if $x\in X$ and $g\in G$
\renewcommand{\arraystretch}{2}
\[\begin{array}{rcl}
	d^1\big(f(gx),gf(x)\big) & = & d^1\big(\psi(e,gx),g\psi(e,x)\big)\\
	& = & d^1\big(\psi(e,gx),\psi(g,gx)\big)  \\
	& \leq & \frac{1}{R}\, (2N+2)(2N+3)\,  d_G(e,g) = \varepsilon {\| g \|}_{d_G}
\end{array}\]
This completes the proof.
\end{proof}

Proposition~\ref{prop:NF-amenable} yields the following corollaries to Theorems~\ref{thmA-asdim}, \ref{thmAn}, and~\ref{thmA}.

\begin{theorem}\label{cor:NF-amenable}
	Let $G$ be a countable group and $\cF$ be a family of subgroups of $G$. 
	If there exists an $N$-$\cF$-amenable action of $G$ on a compact metrizable space $X$ and $\asdim(F)\leq k$ for each $F\in \cF$, then $\asdim(G)\leq N+k$.\qed	
\end{theorem}

\begin{theorem}\label{cor:NF-amenable-general}
	Let $G$ be a countable group, $\cF$ be a family of subgroups of $G$, and $\fC$ be a collection of metric families satisfying Coarse Permanence~(\ref{def:closed-under-embeddings}), Finite Amalgamation Permanence~(\ref{def:amalgam}), and Finite Union Permanence~(\ref{def:finunion}). 
	
	 If there exists an $N$-$\cF$-amenable action of $G$ on a compact metrizable space $X$ and each $F\in \cF$ belongs to $\fC$, then $G$ is $N$-decomposable over $\fC$. If $\fC$ is also stable under weak decomposition, then $G$ is in $\fC$.\qed	
\end{theorem}

Since the collections $\wD$, $\fE$ and $\fH$ all satisfy the assumptions of Theorem~\ref{cor:NF-amenable-general} (see Remark~\ref{rem:collections}), we have the following corollary.

\begin{corollary}\label{cor:deh}
Let $\fC$ be equal to $\wD$, $\fE$ or $\fH$. Let $G$ be a countable group and $\cF$ be a family of subgroups of $G$ such that each $F\in \cF$ belongs to $\fC$.
If there exists a finitely $\cF$-amenable action of $G$ on a compact metrizable space, then $G$ is in $\fC$.	\qed	
\end{corollary}

The main theorem of \cite{Bartels17} tells us that if a countable group $G$ is relatively hyperbolic with respect to peripheral subgroups $P_1,\dots, P_n$, then the action of $G$ on its boundary is finitely $\cP$-amenable, where $\cP$ is the family of subgroups of $G$ that are either virtually cyclic or subconjugated to one of the $P_i$'s. Combining this with Theorem~\ref{cor:NF-amenable-general} yields the following result.

\begin{theorem}\label{thm:rel_hyp}
Let $G$ be a countable group that is relatively hyperbolic with respect to peripheral subgroups $P_1,\dots,P_n$, and let $\fC$ be a collection of metric families satisfying Coarse Permanence, Finite Amalgamation Permanence, and Finite Union Permanence. If $\fC$ contains $P_1,\dots,P_n$ and the infinite cyclic group $\Z$, then $G$ is $N$-decomposable over $\fC$ for some~$N$. If $\fC$ is also stable under weak decomposition, then $G$ is in $\fC$.
\end{theorem}

\begin{proof}
Let $\cP$ be the family of subgroups of $G$ whose members are either virtually cyclic or subconjugated to one of the $P_i$'s. Note that a virtually cyclic group is coarsely equivalent to either $\Z$ or the trivial group and that any two conjugate subgroups of $G$ are coarsely equivalent. Therefore, since $\fC$ satisfies Coarse Permanence, every element of $\cP$ is in $\fC$. The theorem now follows from Theorem~\ref{cor:NF-amenable-general} and  the above mentioned result of Bartels, \cite{Bartels17}, that the action of such a group on its boundary is finitely $\cP$-amenable.
\end{proof}

Theorem~\ref{thm:rel_hyp} is a generalization of Osin's well-known result that a finitely generated relatively hyperbolic group has finite asymptotic dimension if its peripheral subgroups do~\cite[Theorem 1.2]{Osin}. In particular, Theorem~\ref{thm:rel_hyp} applies to the collections $\wD$, $\fE$ and~$\fH$. Other authors have also found generalizations of Osin's result. It is interesting to compare Theorem~\ref{thm:rel_hyp} to Ramras-Ramsey's~\cite[Theorem 3.9]{ramras-ramsey}, where a different combination of permanence properties was employed to extend properties of metric spaces to relatively hyperbolic groups. 
In~\cite{Bartels17}, Bartels obtained the remarkable result that a countable relatively hyperbolic group satisfies the Farrell-Jones Conjecture if its peripheral subgroups~do.


\section{$\catz$ groups}\label{catz-groups}

Recall that a discrete group $G$ is said to act \emph{geometrically} on a metric space $Y$ if
it acts by isometries, the action is properly discontinuous and the quotient $Y/G$ is compact.
A \emph{$\catz$ group} is a countable discrete group that admits a geometric action on some 
finite dimensional $\catz$  space
(see \cite[II.1.1, page 158]{Bridson-Haefliger} for the definition of a $\catz$  space).

\begin{question}\label{catzero}
Is the asymptotic dimension of a $\catz$ group finite?
\end{question}

Nick Wright showed that the asymptotic dimension of a finite dimensional $\catz$ cube complex is bounded above by its geometric dimension,
\cite{Wright}.  Hence, any group that acts geometrically on such a complex has finite asymptotic dimension
thereby providing an abundance of examples of $\catz$ groups for which the answer to Question \ref{catzero} is  affirmative.

Let $G$ be a $\catz$ group.
Any amenable subgroup of $G$ is virtually abelian
\cite[Corollaries B and C]{Adams-Ballmann};
furthermore, there is an upper bound, $r$,  on the rank of all such subgroups,
\cite[Theorem C]{Caprace-Monod}.
Thus, if $\cF$ is the collection of amenable subgroups of $G$, then for all $H \in \cF$,  $\asdim H \leq r$.
One approach to answering Question \ref{catzero} affirmatively would be to attempt to apply our Theorem \ref{cor:NF-amenable}
by constructing a finitely $\cF$-amenable action of $G$ on a suitable compact metrizable space $X$.
Let $Y$ be a finite dimensional $\catz$ space on which $G$ acts geometrically
and let $\partial_\infty Y$ denote the \emph{visual boundary} of $Y$ (see  \cite[II.8]{Bridson-Haefliger}).
While $\partial_\infty Y$ would, at first glance,  appear to be a natural  candidate for the sought after compact space $X$,
it cannot fulfill this role as the  $G$-action on $\partial_\infty Y$ may have fixed points.
Nevertheless, Caprace constructed a set  $\partial_\infty^{\rm \, fine}\,Y$,  which he calls a  \emph{refinement} of $\partial_\infty Y$,
with the property that for each $x \in  \partial_\infty^{\rm fine}\,Y$ the stabilizer $G_x$ is amenable, that is, $G_x \in \cF$,
\cite{Caprace-2009}.

\begin{question}\label{refined-boundary}
Does Caprace's refined boundary, $\partial_\infty^{\rm \, fine}\,Y$, have a compact metrizable topology for which the $G$-action on it  is finitely $\cF$-amenable?
\end{question}

An application of Theorem \ref{cor:NF-amenable} yields the following. 
\begin{proposition}\label{prop:catzero-is-true}
An affirmative answer to Question \ref{refined-boundary} implies an affirmative answer to Question \ref{catzero}, that is, all $\catz$ groups have finite
asymptotic dimension. \qed
\end{proposition}

A theme of  \cite{Bartels-Luck-a, BLR08}, which was further emphasized by Bartels in  \cite{Bartels17} (see his \cite[Theorem 4.3]{Bartels17}),
is  that the Farrell-Jones Conjecture for a group $G$ holds relative to a family $\cF$ if there exists an action of $G$ on a compact,
finite dimensional, contractible ANR (absolute neighborhood retract) satisfying a suitable regularity condition relative to~$\cF$.
In particular, the  Farrell-Jones Conjecture for $\catz$ groups could be approached by 
an affirmative answer to a stronger version of Question \ref{refined-boundary}
where the condition that $Y \cup \partial_\infty^{\rm \, fine}\,Y$ is a  compact,
finite dimensional, contractible ANR is also required.
However,  the issue of finding a suitable boundary for $Y$ was cleverly bypassed by Bartels and L\"uck in  their proof of the Farrell-Jones Conjecture for $\catz$ groups, \cite{Bartels-Luck-a, Bartels-Luck-b},
by making use of homotopy actions and large balls in $Y$.


\bibliographystyle{amsalpha}
\bibliography{NR-F-amenable}

\def\polhk#1{\setbox0=\hbox{#1}{\ooalign{\hidewidth
  \lower1.5ex\hbox{`}\hidewidth\crcr\unhbox0}}}
\providecommand{\bysame}{\leavevmode\hbox to3em{\hrulefill}\thinspace}
\providecommand{\MR}{\relax\ifhmode\unskip\space\fi MR }
\providecommand{\MRhref}[2]{%
  \href{http://www.ams.org/mathscinet-getitem?mr=#1}{#2}
}
\providecommand{\href}[2]{#2}
\begin{thebibliography}{BDLM08}

\bibitem[AB98]{Adams-Ballmann}
Scot Adams and Werner Ballmann, \emph{Amenable isometry groups of {H}adamard
  spaces}, Math. Ann. \textbf{312} (1998), no.~1, 183--195. \MR{1645958}

\bibitem[Bar16]{Bartels1}
Arthur Bartels, \emph{On proofs of the {F}arrell-{J}ones conjecture}, Topology
  and geometric group theory, Springer Proc. Math. Stat., vol. 184, Springer,
  [Cham], 2016, pp.~1--31. \MR{3598160}

\bibitem[Bar17]{Bartels17}
\bysame, \emph{Coarse flow spaces for relatively hyperbolic groups}, Compos.
  Math. \textbf{153} (2017), no.~4, 745--779. \MR{3631229}

\bibitem[Bar18]{Bartels-ICM}
\bysame, \emph{{$K$}-theory and actions on {E}uclidean retracts}, Proceedings
  of the {I}nternational {C}ongress of {M}athematicians---{R}io de {J}aneiro
  2018. {V}ol. {II}. {I}nvited lectures, World Sci. Publ., Hackensack, NJ,
  2018, pp.~1041--1062. \MR{3966799}

\bibitem[BB19]{Bartels-Bestvina}
Arthur Bartels and Mladen Bestvina, \emph{The {F}arrell-{J}ones conjecture for
  mapping class groups}, Invent. Math. \textbf{215} (2019), no.~2, 651--712.
  \MR{3910072}

\bibitem[BD04]{Bell_Dranish1}
G.~Bell and A.~Dranishnikov, \emph{On asymptotic dimension of groups acting on
  trees}, Geom. Dedicata \textbf{103} (2004), 89--101. \MR{2034954
  (2005b:20078)}

\bibitem[BDLM08]{BDLM}
N.~Brodskiy, J.~Dydak, M.~Levin, and A.~Mitra, \emph{A {H}urewicz theorem for
  the {A}ssouad-{N}agata dimension}, J. Lond. Math. Soc. (2) \textbf{77}
  (2008), no.~3, 741--756. \MR{2418302}

\bibitem[BH99]{Bridson-Haefliger}
Martin~R. Bridson and Andr\'e Haefliger, \emph{Metric spaces of non-positive
  curvature}, Grundlehren der Mathematischen Wissenschaften [Fundamental
  Principles of Mathematical Sciences], vol. 319, Springer-Verlag, Berlin,
  1999. \MR{1744486}

\bibitem[BL12a]{Bartels-Luck-a}
Arthur Bartels and Wolfgang L\"uck, \emph{The {B}orel conjecture for hyperbolic
  and $\operatorname{CAT}(0)$-groups}, Ann. of Math. (2) \textbf{175} (2012),
  no.~2, 631--689. \MR{2993750}

\bibitem[BL12b]{Bartels-Luck-b}
\bysame, \emph{Geodesic flow for $\operatorname{CAT}(0)$-groups}, Geom. Topol.
  \textbf{16} (2012), no.~3, 1345--1391. \MR{2967054}

\bibitem[BLR08]{BLR08}
Arthur Bartels, Wolfgang L\"uck, and Holger Reich, \emph{The {$K$}-theoretic
  {F}arrell-{J}ones conjecture for hyperbolic groups}, Invent. Math.
  \textbf{172} (2008), no.~1, 29--70. \MR{2385666}

\bibitem[Cap09]{Caprace-2009}
Pierre-Emmanuel Caprace, \emph{Amenable groups and {H}adamard spaces with a
  totally disconnected isometry group}, Comment. Math. Helv. \textbf{84}
  (2009), no.~2, 437--455. \MR{2495801}

\bibitem[CM13]{Caprace-Monod}
Pierre-Emmanuel Caprace and Nicolas Monod, \emph{Fixed points and amenability
  in non-positive curvature}, Math. Ann. \textbf{356} (2013), no.~4,
  1303--1337. \MR{3072802}

\bibitem[DG03]{Dadarlat_Guentner1}
Marius Dadarlat and Erik Guentner, \emph{Constructions preserving {H}ilbert
  space uniform embeddability of discrete groups}, Trans. Amer. Math. Soc.
  \textbf{355} (2003), no.~8, 3253--3275. \MR{1974686}

\bibitem[Dow52]{Dowker}
C.~H. Dowker, \emph{Topology of metric complexes}, Amer. J. Math. \textbf{74}
  (1952), 555--577. \MR{0048020 (13,965h)}

\bibitem[DS06]{DranishnikovSmith}
A.~Dranishnikov and J.~Smith, \emph{Asymptotic dimension of discrete groups},
  Fund. Math. \textbf{189} (2006), no.~1, 27--34. \MR{2213160}

\bibitem[FP90]{FP}
Rudolf Fritsch and Renzo~A. Piccinini, \emph{Cellular structures in topology},
  Cambridge Studies in Advanced Mathematics, vol.~19, Cambridge University
  Press, Cambridge, 1990. \MR{1074175}

\bibitem[GTY12]{Guentner_Tessera_Yu1}
Erik Guentner, Romain Tessera, and Guoliang Yu, \emph{A notion of geometric
  complexity and its application to topological rigidity}, Invent. Math.
  \textbf{189} (2012), no.~2, 315--357. \MR{2947546}

\bibitem[GTY13]{Guentner_Tessera_Yu2}
\bysame, \emph{Discrete groups with finite decomposition complexity}, Groups
  Geom. Dyn. \textbf{7} (2013), no.~2, 377--402. \MR{3054574}

\bibitem[Gue14]{Guentner}
Erik Guentner, \emph{Permanence in coarse geometry}, Recent progress in general
  topology. {III}, Atlantis Press, Paris, 2014, pp.~507--533. \MR{3205491}

\bibitem[GWY17]{GWY}
Erik Guentner, Rufus Willett, and Guoliang Yu, \emph{Dynamic asymptotic
  dimension: relation to dynamics, topology, coarse geometry, and
  {$C^*$}-algebras}, Math. Ann. \textbf{367} (2017), no.~1-2, 785--829.
  \MR{3606454}

\bibitem[KNR19]{KNR}
Daniel Kasprowski, Andrew Nicas, and David Rosenthal, \emph{Regular finite
  decomposition complexity}, J. Topol. Anal. \textbf{11} (2019), no.~3,
  691--719. \MR{3999517}

\bibitem[L{\"u}c10]{Luck-ICM}
Wolfgang L{\"u}ck, \emph{{$K$}- and {$L$}-theory of group rings}, Proceedings
  of the {I}nternational {C}ongress of {M}athematicians. {V}olume {II},
  Hindustan Book Agency, New Delhi, 2010, pp.~1071--1098. \MR{2827832}

\bibitem[NR18]{Nicas-Rosenthal}
Andrew Nicas and David Rosenthal, \emph{Hyperbolic dimension and decomposition
  complexity}, Topological {M}ethods in {G}roup {T}heory, London Mathematical
  Society Lecture Note Series, Cambridge University Press, Cambridge, 2018,
  Preprint version: arXiv:1509.06437.

\bibitem[Osi05]{Osin}
D.~Osin, \emph{Asymptotic dimension of relatively hyperbolic groups}, Int.
  Math. Res. Not. (2005), no.~35, 2143--2161. \MR{2181790}

\bibitem[RR19]{ramras-ramsey}
Daniel~A. Ramras and Bobby~W. Ramsey, \emph{Extending properties to relatively
  hyperbolic groups}, Kyoto J. Math. \textbf{59} (2019), no.~2, 343--356.
  \MR{3960296}

\bibitem[Saw17]{sawicki}
Damian Sawicki, \emph{On equivariant asymptotic dimension}, Groups Geom. Dyn.
  \textbf{11} (2017), no.~3, 977--1002. \MR{3692903}

\bibitem[STY02]{STY}
G.~Skandalis, J.~L. Tu, and G.~Yu, \emph{The coarse {B}aum-{C}onnes conjecture
  and groupoids}, Topology \textbf{41} (2002), no.~4, 807--834. \MR{1905840}

\bibitem[Wri12]{Wright}
Nick Wright, \emph{Finite asymptotic dimension for {${\rm CAT}(0)$} cube
  complexes}, Geom. Topol. \textbf{16} (2012), no.~1, 527--554. \MR{2916293}

\bibitem[Yu98]{Yu}
Guoliang Yu, \emph{The {N}ovikov conjecture for groups with finite asymptotic
  dimension}, Ann. of Math. (2) \textbf{147} (1998), no.~2, 325--355.
  \MR{1626745 (99k:57072)}

\bibitem[Yu00]{Yu2000}
\bysame, \emph{The coarse {B}aum-{C}onnes conjecture for spaces which admit a
  uniform embedding into {H}ilbert space}, Invent. Math. \textbf{139} (2000),
  no.~1, 201--240. \MR{1728880}

\end{thebibliography}

\end{document}